\documentclass[11pt,reqno]{amsart}
\usepackage{amsmath,amstext,amssymb,amscd}
\usepackage{verbatim}
\usepackage{enumerate}
\usepackage{mathrsfs}
\usepackage[dvipsnames,usenames]{xcolor}

\usepackage{amsthm}

\newtheorem{theorem}{Theorem}[section]
\newtheorem{lemma}[theorem]{Lemma}

\newtheorem{proposition}[theorem]{Proposition}

\theoremstyle{definition}
\newtheorem{definition}[theorem]{Definition}

\theoremstyle{remark}
\newtheorem{remark}[theorem]{Remark}

\numberwithin{equation}{section}

\begin{document}

\newcommand{\sgn}{\operatorname{sgn}}

\def\a{\alpha}
\def\b{\beta}
\def\d{\delta}
\def\g{\gamma}
\def\l{\lambda}
\def\o{\omega}
\def\s{\sigma}
\def\t{\tau}
\def\th{\theta}
\def\r{\rho}
\def\D{\Delta}
\def\G{\Gamma}
\def\O{\Omega}
\def\e{\varepsilon}
\def\p{\phi}
\def\P{\Phi}
\def\S{\Psi}
\def\E{\eta}
\def\m{\mu}
\def\grad{\nabla}
\def\bar{\overline}
\newcommand{\reals}{\mathbb{R}}
\newcommand{\naturals}{\mathbb{N}}
\newcommand{\ints}{\mathbb{Z}}
\newcommand{\complex}{\mathbb{C}}
\newcommand{\rationals}{\mathbb{Q}}
\newcommand{\innerprod}[1]{\left\langle#1\right\rangle}
\newcommand{\norm}[1]{\left\|#1\right\|}
\newcommand{\abs}[1]{\left|#1\right|}

\author[Y. Guo]{Yanqiu Guo}
\address{Department of Mathematics \& Statistics \\Florida International University\\
Miami, Florida 33199, USA} \email{yanguo@fiu.edu}

\author[M. Ilyin]{Michael Ilyin}
\address{Department of Mathematics \& Statistics \\Florida International University\\
Miami, Florida 33199, USA} \email{milyi001@fiu.edu}


\title[Sparse distribution of lattice points in annular regions]
{Sparse distribution of lattice points in annular regions}

\date{May 3, 2024}

\maketitle

\begin{abstract}
This paper is inspired by Richards' work on large gaps between sums of two squares \cite{Richards}.
It is shown in \cite{Richards} that there exist arbitrarily large values of $\lambda$ and $\mu$, where $\mu \geq C \log \lambda$, such that intervals $[\lambda, \,\lambda + \mu ]$ do not contain any sums of two squares. 
Geometrically, these gaps between sums of two squares correspond to annuli in $\mathbb R^2$ that do not contain any integer lattice points. A major objective of this paper is to investigate the sparse distribution of integer lattice points within annular regions in $\mathbb R^2$.
Specifically, we establish the existence of annuli $\{x\in \mathbb R^2: \lambda \leq |x|^2 \leq \lambda + \kappa\}$ with arbitrarily large $\lambda$ and $\kappa \geq C \lambda^s$ for $0<s<\frac{1}{4}$, satisfying that any two integer lattice points within any one of these annuli must be sufficiently far apart. This result is sharp, as such a property ceases to hold at and beyond the threshold $s=\frac{1}{4}$. Furthermore, we extend our analysis to include the sparse distribution of lattice points in spherical shells in $\mathbb R^3$.

\end{abstract}

\section{Introduction}
This article is inspired by Richards' paper on large gaps between integers that can be expressed as sums of two squares \cite{Richards}.
It is proved in \cite{Richards} that large gaps between sums of two squares increase at least logarithmically as an asymptotic rate. 
More precisely, let $s_1$, $s_2, \cdots$ be the sequence, arranged in increasing order, of sums of two squares $x^2 + y^2$, then 
\begin{align} \label{loggap}
\limsup_{n\rightarrow \infty} \frac{s_{n+1} - s_n}{\log s_n} \geq C >0.
\end{align}
In \cite{Richards}, $C=\frac{1}{4}$; but this constant has been improved to $C \approx  0.87$ in \cite{Dietmann} by Dietmann et al. 
The logarithmic-type estimate (\ref{loggap}), established by Richards, represents an improvement over a result by Erd\"os \cite{Erdos}.

Note that, formula (\ref{loggap}) can be interpreted as stating that there exist arbitrarily large values of $n$ for which $s_{n+1}-s_n\geq \alpha \log s_n$, where $\alpha>0$.
In other words, the lower bound of large gaps between sums of two squares increases logarithmically.
Geometrically, these gaps between sums of two squares correspond to annular regions in $\mathbb R^2$ that contain no integer lattice points.
Therefore, Richards' result \cite{Richards} can be restated as follows: there exist arbitrarily large values of $\lambda$ and $\mu$, where $\mu \geq \alpha \log \lambda$,
such that intervals $[\lambda,\,\lambda+\mu]$ do not contain sums of two squares, meaning that there are no integer lattice points in annuli
$\{x\in \mathbb R^2:  \lambda \leq  |x|^2 \leq \lambda + \mu\}$.

On the other hand, there has been research in the literature regarding the upper bound of large gaps between sums of squares. Bambah and Chowla \cite{Bambah} proved that if $\beta > 2\sqrt{2}$, then for all large integers $k$, 
there are integers $u$ and $v$ such that $k \leq u^2 + v^2 < k + \beta k^{\frac{1}{4}}$.
This implies that gaps between sums of two squares have an upper bound of polynomial growth rate.
In particular, $s_{n+1} - s_n  \leq \beta s_n^{\frac{1}{4}}$ for sufficiently large $n$, where $\{s_n\}$ is the sequence of sum of two squares arranged in increasing order. 
Also, Shiu \cite{Shiu} provided a very short proof of Bambah-Chowla theorem. 

By combining the results from    \cite{Bambah} and \cite{Richards}, it can be concluded that large gaps between sums of two squares have both lower and upper bounds. Specifically, there exist arbitrarily large $n$ for which:
\begin{align} \label{bounds}
\alpha \log s_n \leq s_{n+1} - s_n \leq  \beta s_n^{\frac{1}{4}},
\end{align}
for some positive constants $\alpha$ and $\beta$.

Geometrically, gaps between sums of two squares correspond to annuli in $\mathbb R^2$ that contain no integer lattice points, 
and the size of the gap is related to the thickness of the annulus. Motivated by this geometric perspective on gaps between sums of two squares, our study focuses on the sparse distribution of integer lattice points within annular regions. 
Specifically, we aim to identify annuli in $\mathbb R^2$ where any two integer lattice points inside such an annulus are sufficiently distant from each other. 
We anticipate that an annulus containing sparsely distributed lattice points will have greater thickness than one with no lattice points. 
To formalize this, we prove that, for any large distance $d$, there exist arbitrarily large $\lambda$ and $\kappa \geq C \lambda^s$, where $0 < s < \frac{1}{4}$, 
satisfying the condition that any two integer lattice points belonging to the annulus $\{x\in \mathbb R^2: \lambda \leq |x|^2 \leq \lambda + \kappa\}$ must be separated by a distance greater than $d$.
It's essential to note that the polynomial growth rate of the interval's length, i.e., $\kappa \geq C \lambda^s$ with $0 < s < \frac{1}{4}$, is significantly larger than the logarithmic growth rate of the gaps between sums of squares. Our result is sharp in the sense that such a property of sparse lattice point distribution ceases to hold, at and beyond the threshold $s=\frac{1}{4}$.

We also consider three dimensions. There are no large gaps between sums of three squares. Due to the representation of sums of three squares (as seen in, e.g., \cite{Gross}), the gaps between sums of three squares can only be 1, 2 or 3. 
Consequently, if $[m, \,m + \delta]$ does not contain sums of three squares, then $0<\delta<3$.
In other words, if a spherical shell $\{x\in \mathbb R^3: m \leq |x|^2 \leq m+ \delta\}$ does not contain any integer lattice points,
then $0<\delta < 3$. This suggests that a spherical shell that contains no integer lattice points has small thickness. However, in this paper, we show that a spherical shell containing sparsely distributed lattice points can have a more substantial thickness. 
In particular, we establish that, for any large distance $d$, there exist arbitrarily large $m$ and $h \geq C_d \sqrt{\log m}$, satisfying the condition that any two integer lattice points belonging to the spherical shell $\{x\in \mathbb R^3: m \leq |x|^2 \leq m + h\}$ must be separated by a distance greater than $d$. Moreover, we prove that if $h$ reaches the order of $h \sim Cm^{1/8}$, then the property of sparse distribution of lattice points in spherical shells ceases to hold.

Large gaps between sums of squares and the sparse distribution of integer lattice points in annular regions have significant applications in the study of the long-term behavior of dissipative dynamical systems. One example of a dissipative dynamical system, modeled by nonlinear partial differential equations, is the reaction-diffusion equation:\begin{align} \label{reaction}
\partial_t u - \Delta u + f(u) = 0,
\end{align}
where $f(u)$ is a nonlinear term, such as $f(u) = u^3$. When studying this equation on a two-dimensional periodic physical domain, the solution $u$ can be represented as a Fourier series: $u(x,t) = \sum_{k\in \mathbb Z^2} \hat u (k, t)e^{i k \cdot x}$ for $x\in \mathbb T^2 = [0,2\pi]^2$, $t\geq 0$. In the Hilbert space $L^2(\mathbb T^2)$, the Laplace operator $-\Delta$ has eigenfunctions $\{e^{i k \cdot x}\}$ with corresponding eigenvalues $\{k_1^2 + k_2^2\}$, where $k=(k_1, k_2) \in \mathbb Z^2$.
It's important to note that these eigenvalues $\{k_1^2 + k_2^2\}$ are sums of two squares.

In general, a PDE can be considered as a system of infinitely many ODEs with infinitely many unknowns $\hat u (k, t)$, where $k\in \mathbb Z^2$. Some nonlinear dissipative PDEs can be reduced to a system of finitely many ODEs as $t\rightarrow \infty$. This type of finite-dimensional reduction often relies on the existence of large gaps between the eigenvalues of the Laplacian. In fact, large gaps between sums of two squares, as demonstrated in Richards' result \cite{Richards}, can lead to the finite-dimensional reduction for certain dissipative PDEs, such as the reaction-diffusion equation (\ref{reaction}) in a 2D periodic domain.

However, in some cases, large gaps between the eigenvalues of the Laplacian are not available. For instance, if one considers the Laplacian operator acting on functions defined in a 3D periodic domain, the eigenvalues are sums of three squares, which do not exhibit large gaps. In such scenarios, the sparse distribution of integer lattice points in spherical shells can be valuable in reducing a dissipative PDE to a finite-dimensional system at large times. For example, an important work by Mallet-Paret and Sell \cite{MS} shows a finite-dimensional simplification for 3D reaction-diffusion equations using the sparse distribution of lattice points in spherical shells.

\vspace{0.2 in}

\section{Statements of main results}

Our first result is concerned with the sparse distribution of lattice points in annuli in $\mathbb R^2$. 

\begin{theorem}  \label{thm2D}
Assume $0<s<\frac{1}{4}$. Given that $d \geq 1$ and $C>0$. There exist arbitrarily large $\lambda \in \mathbb R^+$ and $\kappa \geq C \lambda^s$, such that,  
any two lattice points $k,\ell \in \mathbb Z^2$ that belong to the annulus $\{x\in \mathbb R^2: \lambda \leq |x|^2 \leq \lambda + \kappa\}$ must satisfy $|k-\ell|>d$.
\end{theorem}

\begin{remark}
The annulus $\{x\in \mathbb R^2: \lambda \leq |x|^2 \leq \lambda + \kappa\}$ described in the theorem may contain either no lattice points, only one lattice point, or multiple lattice points. However, when there are multiple lattice points within such an annulus, it is ensured that the distance between any two lattice points is sufficiently large.
\end{remark}

The following result shows the optimality of Theorem \ref{thm2D}.

\begin{proposition} \label{prop-new}
Let $\alpha>4\sqrt{2}$. For any sufficiently large $\lambda \in \mathbb R^+$, there exist two integer lattice points with a distance of 1, belonging to the annulus $\{x\in \mathbb R^2:  \lambda \leq |x|^2 \leq \lambda +  \alpha \lambda^{1/4}\}$.
\end{proposition}

\begin{remark}
In Theorem \ref{thm2D}, the value $\kappa$ reflects the ``width" of the annulus where lattice points are sparsely distributed, with $\sqrt{\lambda}$ being the inner radius of the annulus. These parameters are related by the inequality $\kappa \geq C \lambda^s$ for $0<s<\frac{1}{4}$. Theorem \ref{thm2D} is sharp because, if $s$ reaches the threshold power of $\frac{1}{4}$, then the property of sparse distribution of lattice points in annuli ceases to hold. It can be seen from Proposition \ref{prop-new} that for any sufficiently large $\lambda$, an annulus $\{x\in \mathbb R^2:  \lambda \leq |x|^2 \leq \lambda +  \alpha \lambda^{1/4}\}$, with $\alpha>4\sqrt{2}$, must contain at least two integer lattice points separated by a small distance of 1. Furthermore, it is important to notice that the power $1/4$ also appears in a classical result by Bambah and Chowla \cite{Bambah}, regarding the upper bound of large gaps between sums of two squares. It fact, Bambah and Chowla's result can be stated geometrically as follows: any annular region $\{x\in \mathbb R^2:  \lambda \leq |x|^2 \leq \lambda + \beta \lambda^{1/4}\}$, where $\beta>2\sqrt{2}$, must contain at least one integer lattice point,
for any large $\lambda$. By comparison, in Bambah and Chowla's result, an annulus contains at least one lattice point, whereas, in Proposition \ref{prop-new}, an annulus contains at least two lattice points separated by a distance of 1.
\end{remark}

\vspace{0.1 in}

Our next result is about the sparse distribution of lattice points in spherical shells in $\mathbb R^3$. 

\begin{theorem} \label{thm3D}
Let $d \geq 1$. There exist arbitrarily large $m\in \mathbb R^+$ and $h \geq C_d \sqrt{\log m}$ for some constant $C_d$ depending only on $d$, such that,  
any two lattice points $k,\ell \in \mathbb Z^3$ that belong to the spherical shell 
$\{x\in \mathbb R^3: m \leq |x|^2 \leq m + h\}$ must satisfy $|k-\ell| > d$.
\end{theorem}

\vspace{0.1 in}

The next result complements Theorem \ref{thm3D}. It provides a sufficient condition under which the sparse distribution of lattice points in spherical shells does not occur. 

\begin{proposition} \label{prop-new2}
Let $C>4\sqrt[4]{8}$. For any sufficiently large $m \in \mathbb R^+$, there exist two integer lattice points with a distance of 1, belonging to the spherical shell $\{x\in \mathbb R^3:  m \leq |x|^2 \leq m + C m^{1/8}\}$.
\end{proposition}

\begin{remark}
\noindent
\begin{enumerate}[(i)]
\item In Theorem \ref{thm3D}, the value $h$ exhibits the ``thickness" of the spherical shell where the lattice points are sparsely distributed. 
Theorem \ref{thm3D} states that $h$ has a lower bound $C_d\sqrt{\log m}$. On the other hand, Proposition \ref{prop-new2} provides that $h$ has an upper bound $C m^{1/8}$ with $C>4\sqrt[4]{8}$. At and beyond this threshold, spherical shells no longer contain sparsely distributed lattice points. 
An open problem is to find an optimal asymptotic estimate for $h$, similar to what we have obtained for the 2D case. It is also worth mentioning that the upper and lower bounds of $h$, namely, $C_d \sqrt{\log m} \leq h < C m^{1/8}$, align with the upper and lower bounds of large gaps between sums of squares specified in (\ref{bounds}).
\item  Because the gaps between sums of three squares are at most 3, there exists at least one integer lattice point belonging to the spherical shell $\{x\in \mathbb R^3:  m \leq |x|^2 \leq m + 3\}$ for any real number $m\geq 0$. This fact complements Proposition \ref{prop-new2}.
\end{enumerate}
\end{remark}

\vspace{0.2 in}

\section{Proof of Theorem \ref{thm2D} and Proposition \ref{prop-new}}

\subsection{Proof of Theorem \ref{thm2D}}
We provide a proof of Theorem \ref{thm2D}, which asserts the sparse distribution of lattice points in annuli in $\mathbb R^2$. Before presenting the proof, we introduce a notation for asymptotic equivalence: given functions $f(x)$ and $g(x)$, we denote
\begin{align}
f(x) \sim g(x) \;\; \text{to mean}  \;\;  \lim_{x\rightarrow \infty} \frac{f(x)}{g(x)} = 1.
\end{align}

\vspace{0.1 in}

\begin{proof}
We draw ideas from \cite{MS}. 

Assume the distance $d \geq 1$. Also, we fix an arbitrary constant $C>0$.

Consider the family of disjoint annuli in $\mathbb R^2$:
\begin{align*}
N_m^{\mu} = \{x\in \mathbb R^2:  \mu+ m \kappa < |x|^2 \leq \mu + (m+1) \kappa\},
\end{align*}
where $m \in \mathbb Z$ with $0\leq m \leq  J = \lfloor \mu^{1/2} \rfloor $.
We set
$$\kappa = C\mu^s, \;\;\text{where}\;\;  0<s<\frac{1}{4}.$$

We aim to show that, for sufficiently large $\mu$, there exists $m\in [0,J]$ such that $N_m^{\mu}$ does not contain a pair of lattice points with a distance less than or equal to $d$.

The union of these annuli is denoted as
\begin{align*}
N^{\mu} = \bigcup_{m=0}^{J} N_m^{\mu} =
\{x \in \mathbb R^2: \;  \mu < |x|^2 \leq \mu + (J+1)\kappa\}.
\end{align*}
Notice that $N^{\mu}$ is also an annulus.

We will estimate the thickness of the annulus $N^{\mu}$:
\begin{align}  \label{Nu1}
\text{thickness of} \,N^{\mu} &= \sqrt{\mu + (J+1)\kappa} - \sqrt{\mu}.
\end{align} 
Using $J=  \lfloor \mu^{1/2} \rfloor $ and $\kappa =   C\mu^s$ for $0<s<\frac{1}{4}$, a simple calculation shows that, as $\mu\rightarrow \infty$,
\begin{align} \label{thickN}
\text{thickness of} \,N^{\mu} \sim \frac{1}{2}  C\mu^s, \;\text{namely},\; \lim_{\mu\rightarrow \infty} \frac{\text{thickness of} \,N^{\mu} }{\frac{1}{2} C\mu^s} = 1.
\end{align}

Suppose there exist lattice points $k, \ell  \in \mathbb Z^2$ such that
\begin{align*}
k, \ell  \in N^{\mu}_m  \;\;\text{with} \;\; 0<|k - \ell| \leq  d,
\end{align*}
for some $m \in [0,J]$. Let $j= \ell-k$, then $0<|j| \leq d$. Since
\begin{align*}
|\ell|^2 = |k|^2 + 2 k \cdot j + |j|^2,
\end{align*}
then
\begin{align*}
|k \cdot j|\leq \frac{1}{2} \left| |k|^2 - |\ell|^2 \right| + \frac{1}{2} |j|^2 < \frac{1}{2}\kappa + \frac{1}{2} d^2,
\end{align*}
because $k, \ell  \in N^{\mu}_m$. Since $k$ and $\ell$ are interchangeable, we have also
$|\ell \cdot j| < \frac{1}{2} \kappa + \frac{1}{2} d^2$. Therefore, the lattice points $k$ and $\ell$ belong to the strip 
\begin{align} \label{strip}
S_j^{\mu} = \{x \in \mathbb R^2: |x \cdot j|< \frac{1}{2}\kappa + \frac{1}{2}d^2\},
\end{align}
for some $j\in \mathbb Z^2$ satisfying $0<|j| \leq d$. The strip $S_j$ is symmetric about the origin, and also symmetric about the straight line $x \cdot j =0$.

We denote $S^{\mu}$ as the finite union:
\begin{align*} 
S^{\mu} = \bigcup_{\substack{j\in \mathbb Z^2 \\ 0<|j| \leq d}} S^{\mu}_j.
\end{align*}
Note that the set $S^{\mu}$ contains all pairs of lattice points $(k,\ell)$ with distance less than or equal to $d$ belonging to an annulus $N^{\mu}_m$ for some $m \in [0,J]$.
In other words,
\begin{align} \label{Smusup}
S^{\mu} \supset \left\{  k, \ell \in \mathbb Z^2:  \,  0< |k-\ell| \leq d \;\text{with} \; k, \ell \in N^{\mu}_m \; \text{for some} \; m \in [0,J] \right\}.
\end{align}
We remark that constructing set $S^{\mu}$ is crucial for this proof.

Using (\ref{strip}), we observe that, for sufficiently large $\mu$,
\begin{align} \label{widS}
\text{the width of} \; S^{\mu}_j \leq 2\kappa = 2  C\mu^s.  
\end{align}

Also, as $\mu \rightarrow \infty$, asymptotically,
\begin{align} \label{measSN}
\text{meas}(S_j^{\mu} \cap N^{\mu}) \sim 2 (\text{width of $S^{\mu}_j$})(\text{thickness of $N^{\mu}$}).
\end{align}
Combining (\ref{thickN}), (\ref{widS}) and (\ref{measSN}), it follows that, for sufficiently large $\mu$, 
\begin{align} \label{measSN2}
\text{meas}(S_j^{\mu} \cap N^{\mu})  \leq  3C^2 \mu^{2s}.
\end{align}

Notice that the region $S_j^{\mu} \cap N^{\mu}$ is the intersection of an annulus and a strip symmetric about the origin.
As $\mu\rightarrow \infty$, both the thickness of the annulus and the width of the strip approach infinity. Therefore, as $\mu\rightarrow \infty$, the number of integer lattice points in the region $S_j^{\mu} \cap N^{\mu}$ is asymptotically equal to its area:
\begin{align} \label{cardSN}
\text{card}(S^{\mu}_j \cap N^{\mu} \cap \mathbb Z^2) \sim \text{meas} (S^{\mu}_j \cap N^{\mu}).
\end{align} 
By (\ref{measSN2}) and (\ref{cardSN}), for sufficiently large $\mu$, 
\begin{align}   \label{measSN3}  
\text{card}(S^{\mu}_j \cap N^{\mu} \cap \mathbb Z^2)  \leq 4C^2 \mu^{2s}.
\end{align}
Moreover, since $S^{\mu} = \bigcup_{\substack{j\in \mathbb Z^2 \\ 0<|j| \leq d}} S^{\mu}_j$ is a finite union, then by (\ref{measSN3}), for sufficiently large $\mu$,
 \begin{align}   \label{cont}
 \text{card}(S^{\mu} \cap N^{\mu} \cap \mathbb Z^2) \leq 16 d^2C^2 \mu^{2s}.
\end{align}

If each of the disjoint sets $S^{\mu} \cap N^{\mu}_m \cap \mathbb Z^2$ is not empty for every $m \in [0,J]$, where $J=\lfloor \mu^{1/2} \rfloor$,
then $\text{card}(S^{\mu} \cap N^{\mu} \cap \mathbb Z^2)$ will grow at least as fast as $\mu^{1/2}$ as $\mu\rightarrow \infty$, which contradicts (\ref{cont})
because $0<2s<1/2$. Therefore, for sufficiently large $\mu$, there exists $m_0\in [0,J]$ such that the set $S^{\mu} \cap N^{\mu}_{m_0} \cap \mathbb Z^2$ is empty. 
Then, we conclude from (\ref{Smusup}) that the annulus $N^{\mu}_{m_0}$ does not contain two lattice points with a distance less than or equal to $d$.
Denote $\lambda = \mu + m_0\kappa$, then 
\begin{align} \label{Nm0}
N_{m_0}^{\mu} &= \{x\in \mathbb R^2:  \mu+ m_0 \kappa < |x|^2 \leq \mu + (m_0+1) \kappa\} \notag\\
&= \{x\in \mathbb R^2:  \lambda < |x|^2 \leq \lambda + \kappa\}.
\end{align}
Notice that $\lim_{\mu \rightarrow \infty} \frac{\lambda}{\mu}=1$ and $\kappa = C \mu^s$, and therefore, 
$\lim_{\lambda \rightarrow \infty} \frac{C\lambda^s}{\kappa}=1$, where $0<s<1/4$.
Hence, we have $\kappa \geq \frac{1}{2}C \lambda^s$, for sufficiently large $\lambda$. 
Also, the half open annulus given in (\ref{Nm0}) can be easily adjusted to a closed annulus. 
\end{proof}

\vspace{0.1 in}

\subsection{Proof of Proposition \ref{prop-new}}

\begin{proof}
Let $\alpha>4\sqrt{2}$. Consider an integer $m=\lfloor \lambda^{1/2} \rfloor$.
For any sufficiently large $\lambda$, we claim that
$\sqrt{\lambda +\alpha \lambda^{1/4} - m^2} - \sqrt{\lambda - m^2} >2$.
Indeed,
\begin{align} \label{op-2}
&\sqrt{\lambda + \alpha \lambda^{1/4} - m^2} - \sqrt{\lambda - m^2} 
 = \frac{\alpha \lambda^{1/4}}{\sqrt{\lambda + \alpha \lambda^{1/4} - m^2} + \sqrt{\lambda - m^2}} \notag\\
&\geq  \frac{\alpha \lambda^{1/4}}{\sqrt{\lambda +\alpha \lambda^{1/4} - (\lambda^{1/2}-1)^2} + \sqrt{\lambda - ({\lambda^{1/2}-1})^2}} \notag\\
&=  \frac{\alpha \lambda^{1/4}}{\sqrt{\alpha \lambda^{1/4} + 2\lambda^{1/2}-1} + \sqrt{2\lambda^{1/2}-1}} \sim \frac{\alpha }{2\sqrt{2}} > 2, \;\;\text{as}\; \lambda \rightarrow \infty.
\end{align}
Thus, there exists a positive integer $n$ such that
\begin{align}
\sqrt{\lambda - m^2} \leq n < n+1 \leq   \sqrt{\lambda + \alpha \lambda^{1/4} - m^2}.
 \end{align}
It follows that the integer lattice points $(m,n)$ and $(m,n+1)$ both belong to the annulus $\{x\in \mathbb R^2:  \lambda \leq |x|^2 \leq \lambda +\alpha \lambda^{1/4}\}$.
\end{proof}

\vspace{0.2 in}

\section{Proof of Theorem \ref{thm3D} and Proposition \ref{prop-new2}}

This section is devoted to the proof of Theorem \ref{thm3D}. It is about the sparse distribution of lattice points in spherical shells in $\mathbb R^3$. Before presenting the proof, we introduce some concepts in number theory and state some lemmas.

\subsection{Definitions and lemmas} 

\noindent

Let us recall the definitions of Legendre's symbol and Kronecker's symbol. One may refer to classic books \cite{Hardy, Hua}. The notation $p\nmid m$ means that $p$ does not divide $m$.
\begin{definition}
Given an odd prime $p>0$ and an integer $m$ with $p\nmid m$. The integer $m$ is called a \emph{quadratic residue} mod $p$ if $m\equiv k^2$ (mod $p$) for some $k\in \mathbb Z$. 
If the equation $m\equiv k^2$ (mod $p$) has no solution $k$, then $m$ is called a \emph{quadratic non-residue} mod $p$. The \emph{Legendre's symbol} is defined as
\begin{align} \label{Lego}
\left( \frac{m}{p} \right) =
\begin{cases}
1    & \text{if} \; m \;\text{is a quadratic residue mod} \; p \, , \\
-1   & \text{if}  \; m \;\text{is a quadratic non-residue mod} \;  p\,.
\end{cases}
\end{align} 
\end{definition}

\vspace{0.1 in}

In the theory of quadratic forms, the discriminant $d= b^2 - 4ac$ is considered, which implies $d=0$ or 1 (mod 4). Under this scenario, we define the Kronecker's symbol as follows.
\begin{definition}
Assume $d \equiv 0$ or 1 (mod 4) and $d \not=0$. Let $n>0$ with the prime factorization $n= \prod_{j=1}^k p_j$. 
Assume gcd$(d,n) =1$. The \emph{Kronecker's symbol} $\left(\frac{d}{n} \right)$ is defined as
\begin{align} \label{Jacobi}
\left( \frac{d}{n} \right)   = \prod_{j=1}^k \left( \frac{d}{p_j} \right)  ,
\end{align}
where $\left(\frac{d}{p}\right)$ is the Legendre's symbol for odd prime $p$ with $p \nmid d$, and 
\begin{align}  \label{Jaco2}
\left( \frac{d}{2} \right) =
\begin{cases}
1    & \text{if} \; d \equiv 1 \;(\text{mod} \;8), \\
-1   & \text{if}  \; d \equiv 5\;(\text{mod} \;8).
\end{cases}
\end{align}
The Kronecker's symbol $\left(\frac{d}{n} \right)$ can be extended to negative values of $n$ by using 
$\left(\frac{d}{-m} \right) = \left(\frac{d}{-1} \right) \left(\frac{d}{m} \right)$ with $\left(\frac{d}{-1} \right) = 1$ when $d>0$;  $\left(\frac{d}{-1} \right) = -1$ when $d<0$.
\end{definition}

\vspace{0.1 in}

Kronecker's symbol has the property
\begin{align} \label{Jaco-p}
\left( \frac{d}{n_1}\right) =  \left( \frac{d}{n_2}\right),  \;\; \text{if} \;  n_1 \equiv n_2 \, ( \text{mod}  \;d) \;\;\text{and} \;\; d \equiv 0 \; \text{or} \; 1 \,(\text{mod} \; 4),
\end{align}
provided $d \not = 0$, $n_j \not = 0$, and gcd $(d, n_j)=1$, for $j=1,2$. 

\vspace{0.1 in}

Lemma \ref{lem1} and Lemma \ref{lem2} have been proved by Mallet-Paret and Sell in \cite{MS}.

\begin{lemma} \label{lem1} \emph{(Mallet-Paret and Sell \cite{MS})}
Let $D\subset \mathbb{Z}$ be a finite nonempty set of integers $d \equiv 0$ or 1 (mod 4), with the property that $\prod_{d\in A} d$ is not a perfect square whenever $%
A\subset D$ has odd cardinality. Then there exists an integer $r\not = 0$ such that
\begin{align*}
\emph{gcd}(d,r) = 1  \;\;   \text{and}  \;\;  \left( \frac{d}{r}\right) = -1
\end{align*}
for each $d\in D$.
\end{lemma}

\vspace{0.1 in}

Before stating the next lemma, we introduce some notations. Let $p>0$ be a prime. We use notations $p\,|_e\,  n $ and $p\,|_o\, n$ to represent that $p$ divides $n$ an even or odd number of times, respectively.
More precisely, we write 
$$p\,|_e\, n$$
to mean either $n = p^{\alpha} m$ where $\alpha$ is even and $p \nmid m$, or else $n=0$. Note that $\alpha=0$ is permitted, so $p \, |_e \, n$ holds if $p \nmid n$. Similarly, we write 
$$p\,|_o\, n$$
to mean $n = p^{\alpha} m$, where $\alpha$ is odd and $p \nmid m$.

\vspace{0.1 in}

\begin{lemma} \label{lem2}  \emph{(Mallet-Paret and Sell \cite{MS})}
Consider a quadratic form $T(k_1,k_2)=ak_1^2 + b k_1 k_2 + c k_2^2$ with integer coefficients and discriminant $d=b^2 - 4 ac$,
and let $p$ be a prime satisfying $p\nmid d$ and $\left( \frac{d}{p}\right) = -1$. Then
\begin{align*}
p\,|_e \, T(k_1,k_2)
\end{align*}
for any $k_1$, $k_2\in \mathbb Z$.
\end{lemma}

\vspace{0.1 in}

\begin{remark}
In \cite{MS}, Lemma \ref{lem2} was proved for odd prime $p$. The same conclusion holds for the case $p=2$, and we briefly show the proof as follows. Let $p=2$. 
Since $\left( \frac{d}{2}\right) = -1$, then $d \equiv 5$ (mod 8) by (\ref{Jaco2}).
Then, $d^2 \equiv 9$ (mod 16). Since $2\nmid d$ and $d= b^2 - 4ac$, we obtain that $b$ is odd. Then, $a$ and $c$ must both be odd. In fact, if either $a$ or $c$ is even, 
then $d^2 = (b^2 - 4ac)^2 \equiv 1$ (mod 16), contradicting that $d^2 \equiv 9$ (mod 16). Therefore, we conclude that all of $a,b,c$ are odd, namely, coefficients of quadratic form $T(k_1,k_2)$ are all odd. 
Therefore, if $2$ divides $T(k_1,k_2)$, then 2 divides both $k_1$ and $k_2$, and thus $4$ divides $T(k_1,k_2)$. Repeatedly factoring out 4 gives that $2\,|_e \, T(k_1,k_2)$.
\end{remark}

\vspace{0.1 in}

The following lemma is motivated by the works in \cite{Richards,MS}. It's essential to emphasize that we present a logarithmic-type estimate for the length of an interval, satisfying the condition that a family of quadratic forms does not take values within the interval. It is a generalization of Richards' result in \cite{Richards} to a family of quadratic forms. This logarithmic-type estimate plays a critical role in justifying Theorem \ref{thm3D}, which concerns the thickness of spherical shells containing sparsely distributed lattice points.

\begin{lemma}
\label{thm1} Let $D\subset \mathbb{Z}$ be a finite nonempty set of integers 
$d \equiv 0$ or 1 (mod 4), with the property that $\prod_{d\in A} d$ is not a perfect square whenever $%
A\subset D$ has odd cardinality. There exist arbitrarily large $m$ and $h\geq C \log m$ for some constant $C>0$ that depends solely on $D$,
satisfying: if $T$ is any quadratic form 
\begin{align*}
T(k_1,k_2)=ak_1^2 + bk_1 k_2 + ck_2^2, \quad a, b, c\in \mathbb{Z},
\end{align*}
with discriminant $d=b^2 - 4 ac \in D$, then 
\begin{align*}
T(k_1,k_2) \not \in [m,m+h] \text{\;\;for each\;\;} k_1, k_2 \in \mathbb{Z}.
\end{align*}
\end{lemma}

\begin{remark}  \label{remk1}
If $D$ contains only negative integers, then obviously $\prod_{d\in A} d$ is not a perfect square whenever $A\subset D$ has odd cardinality. 
\end{remark}

\smallskip

\begin{proof}
The argument adopts ideas from \cite{Richards, MS}. Thanks to Lemma \ref{lem1}, there exists $r\not=0$ such that 
\begin{equation}
\text{gcd\ }(d,r)=1\text{\ \ \ and\ \ \ }\left( \frac{d}{r}\right) =-1,\text{%
\ \ for each\ \ }d\in D.  \label{ms4}
\end{equation}%
Define 
\begin{align} \label{def-delta}
\delta :=\text{\ lcm\ }\{|d|:d\in D\}.
\end{align}
Note, (\ref{ms4}) and (\ref{def-delta}) imply that 
$\text{gcd\ }(\delta,r)=1$. Let $h>0$ and set 
\begin{equation} 
A:=\sup_{0\leq j\leq h}|r+\delta j|.  \label{ms8}
\end{equation}%
Define $P$ be the product 
\begin{equation}
P:=\prod p^{1+\alpha }  \label{ms7}
\end{equation}%
where the product is taken over all primes $p$ with 
\begin{equation} \label{ms7'}
p\nmid \delta \text{\ \ and\ \ }p^{\alpha }\leq A<p^{1+\alpha }\text{\ \ for
some integer\ \ }\alpha >0.
\end{equation}%
Because of (\ref{ms7}) and (\ref{ms7'}), we have gcd$(P ,\delta)=1$. 
Then, by Bezout's identity, there exists an integer $m\in \lbrack 1,P]$ satisfying 
\begin{equation}
\delta m\equiv r\;(\text{mod}\;P).  \label{ms6}
\end{equation}%

We argue that $h\geq C\log m$, if $h$ is sufficiently large. Indeed, the number of primes $p\leq A$ is asymptotic to $\frac{A}{%
\log A}$. Therefore, for $A$ sufficiently large, the number of primes $p\leq A$ is less than 
$\frac{2A}{\log A}$.

By (\ref{ms7}) and (\ref{ms7'}), we obtain that 
\begin{equation} \label{ms9}
P   =  \prod p^{1+\alpha } \leq   \prod p^{2\alpha} \leq    \prod A^2 \leq
A^{\frac{4A}{\log A}}=e^{4A}.
\end{equation}%

Due to (\ref{ms8}), we have, for sufficiently large $h$,
\begin{align}  \label{ms10}
A \leq |r| + \delta h \leq 2 \delta h.
\end{align}

Because of (\ref{ms9}) and (\ref{ms10}) together with the fact $m\leq P$, we conclude that, for sufficiently large $h$,
\begin{align}   \label{ms11}
m \leq e^{8 \delta h}.
\end{align}
Inequality (\ref{ms11}) can be written as
$h \geq  \frac{1}{8\delta} \log m$, for sufficiently large $h$.

We claim that $T(k_{1},k_{2})\not\in [m,m+h]$\ \ for any\ \ $k_{1},k_{2}\in \mathbb{Z}$.

Indeed, thanks to Lemma \ref{lem2}, it is sufficient to show that 
whenever $0\leq j \leq h$ and $d = b^2 - 4ac \in D$, there exists a prime number $p$ satisfying 
\begin{align} \label{claim}
p\nmid d, \;\;\;   \left( \frac{d}{p}\right) = -1, \;\; \text{and} \;\;  p\,|_o\, (m+j).
\end{align}

We take an arbitrary integer $j \in [0,h]$. Note, $r+\delta j \equiv r$ (mod $d$) due to (\ref{def-delta}). Thus, we obtain
from (\ref{Jaco-p}) and (\ref{ms4}) that
\begin{align} \label{-1}
\left(\frac{d}{r+ \delta j} \right) = \left(\frac{d}{r} \right) = -1.
\end{align}
Since gcd$(d,r)=1$ and $d$ divides $\delta$, we see that gcd$(d,r+\delta j)=1$.

We write the prime factorization for $r+ \delta j = \prod p^a $
and use (\ref{Jacobi}) to obtain that
$\left(\frac{d}{r+ \delta j} \right) = \prod \left( \frac{d}{p} \right)^a = -1$. 
It follows that there exists a prime $p\nmid d$ with $\left( \frac{d}{p}\right) = -1$ satisfying 
\begin{align} \label{po}
p\, |_o \, (r + \delta j),
\end{align}
namely, $p$ divides $r+ \delta j$ an odd number of times. 

Since $\text{gcd\ }(\delta,r)=1$ and using (\ref{po}),
we have $p\nmid \delta$. 
Then, because of (\ref{ms8}), (\ref{ms7}) and (\ref{ms7'}), we see that $p$, as a factor of $P$, occurs with a greater multiplicity than as a factor of $r + \delta j$. Moreover, by (\ref{ms6}), 
\begin{align*}
\delta (m+j)\equiv r+\delta j\;(\text{mod}\;P).
\end{align*}
As a result, $p$ divides $\delta (m+j)$ and $r+\delta j$ the same number of times. 
Then, due to (\ref{po}) and $p\nmid \delta$, it follows that $p\,|_o\, (m+j)$ as claimed in (\ref{claim}).
\end{proof}

\vspace{0.1 in}

\begin{proposition} \label{prop1}
Consider a quadratic function $T(k_1, k_2) = a k_1^2 + b k_1 k_2 + c k_2^2 + s k_1 + t k_2 + u$ for $k_1,k_2\in \mathbb Z$, 
where coefficients $a,b,c,s,t,u \in \mathbb Q$ such that $b^2 - 4ac \not = 0$. 
Then there exist $\xi_1, \xi_2, \xi_3 \in \mathbb Q$
such that $T(x_1 + \xi_1, x_2 + \xi_2) - \xi_3 = a x_1^2 + b x_1 x_2 + c x_2^2$ for all $x_1, x_2 \in \mathbb Q$.
\end{proposition}

\begin{proof}
Let us explicitly find $\xi_1, \xi_2, \xi_3 \in \mathbb Q$ such that the following equality holds for all $x_1, x_2 \in \mathbb Q$. Consider 
\begin{align*}
&T(x_1 + \xi_1, x_2 + \xi_2) - \xi_3  
= a x_1^2 + b x_1 x_2 + c x_2^2 + x_1(2 a \xi_1 + b \xi_2 + s) + x_2(2 c \xi_2 + b \xi_1 + t)  \notag\\
&+(a \xi_1^2 + b \xi_1 \xi_2 + c\xi_2^2 + s\xi_1 + t\xi_2 + u -\xi_3) = a x_1^2 + b x_1 x_2 + c x_2^2.
\end{align*}
Therefore,
\begin{align*}
\begin{cases}
2 a \xi_1 + b \xi_2 + s =0 \\
2 c \xi_2 + b \xi_1 + t = 0 \\
a \xi_1^2 + b \xi_1 \xi_2 + c\xi_2^2 + s\xi_1 + t\xi_2 + u -\xi_3 = 0\,.
\end{cases}
\end{align*}
Hence, 
\begin{align*}
\begin{pmatrix}
\xi_1  \\
\xi_2  
\end{pmatrix}
=\begin{pmatrix}
2a & b \\
b & 2c 
\end{pmatrix}^{-1}
\begin{pmatrix}
-s  \\
-t  
\end{pmatrix}
= \frac{1}{4ac - b^2} 
\begin{pmatrix}
2c & -b \\
-b & 2a 
\end{pmatrix}
\begin{pmatrix}
-s  \\
-t  
\end{pmatrix}
= \frac{1}{4ac - b^2} 
\begin{pmatrix}
-2cs + bt  \\
bs - 2 a t
\end{pmatrix}.
\end{align*}
It follows that
\begin{align} \label{xi1-2}
\xi_1 = \frac{bt - 2cs }{4ac - b^2} , \;\;\;  \xi_2 = \frac{bs - 2at}{4ac - b^2}.
\end{align}
With values of $\xi_1$ and $\xi_2$, we can find the value of $\xi_3$:
\begin{align} \label{xi3}
\xi_3 = a \xi_1^2 + b \xi_1 \xi_2 + c\xi_2^2 + s\xi_1 + t\xi_2 + u.
\end{align}
\end{proof}

\vspace{0.1 in}

\subsection{Proof of Theorem \ref{thm3D}}

\noindent

With the preparations above, now we are ready to prove Theorem 2.5. 
\begin{proof}
We adopt ideas from \cite{MS}. 
Let us fix a distance $d \geq 1$. Consider a spherical shell in $\mathbb R^3$:
\begin{align*}
N=
\{x \in \mathbb R^3: \;  m\leq  |x|^2 \leq m+h\}.
\end{align*}
Suppose there exist lattice points $k, \ell  \in \mathbb Z^3$ such that
\begin{align} \label{lessd}
k, \ell  \in N \;\;\text{with} \;\; 0<|k - \ell| \leq d.
\end{align}

Let $j= \ell-k$, then $0<|j| \leq d$. Thus, $|\ell|^2 = |k|^2 + |j|^2 + 2 k \cdot j$. It follows that
\begin{align*}
|k \cdot j|\leq \frac{1}{2} \left| |k|^2 - |\ell|^2 \right| + \frac{1}{2} |j|^2 \leq \frac{1}{2}h + \frac{1}{2} d^2,
\end{align*}
because $k,\ell \in N$.

We denote $n = k \cdot j$. Then $n\in \mathbb Z$ satisfying
\begin{align} \label{nbound}
|n| = |k \cdot j| \leq \frac{1}{2}h + \frac{1}{2} d^2.
\end{align}

Since $j=(j_1, j_2,j_3) \not = 0$, without loss of generality, we assume $j_3 \not =0$. For $k=(k_1,k_2,k_3)$, solving $n = k \cdot j= k_1 j_1 + k_2 j_2 + k_3 j_3$, we obtain
\begin{align*}
k_3 =  j_3^{-1} (n - k_1j_1 -k_2 j_2). 
\end{align*}
Thus,
\begin{align} \label{Tjn}
|k|^2 &= k_1^2 + k_2^2 + j_3^{-2} ( k_1j_1 + k_2 j_2 -n)^2  \notag\\
&=  j_3^{-2} \left[   (j_1^2 + j_3^2) k_1^2 + (2j_1 j_2 ) k_1 k_2 +     (j_2^2 + j_3^2) k_2^2    - 2n  j_1 k_1 - 2n  j_2 k_2  + n^2  )  \right] \notag\\
&=: T_{j,n}(k_1,k_2),
\end{align}
where $j\in \mathbb Z^3$, $n\in \mathbb Z$ such that $0<|j| \leq d$ and $|n| \leq \frac{1}{2}h + \frac{1}{2} d^2$.  Since $k\in N$, we know that
\begin{align}  \label{Tjnkk}
T_{j,n}(k_1,k_2) \in [m, m+h].
\end{align}

We remark that the function $T_{j,n}$ defined in (\ref{Tjn}) is a function of the form
$T_{j,n}(k_1,k_2) = ak_1^2 + b k_1 k_2 + ck_2^2 + s k_1 + t k_2 + u$
where coefficients $a,b,c,s,t,u\in \mathbb Q$ and $b^2 - 4ac <0$. Thanks to Proposition \ref{prop1}, there exist rational numbers $\xi_1$, $\xi_2$ and $\xi_3$ such that
\begin{align} \label{Txx}
T_{j,n} (x_1 + \xi_1, x_2+\xi_2) - \xi_3 =  j_3^{-2}   (j_1^2 + j_3^2) x_1^2 + (2 j_3^{-2}    j_1 j_2   ) x_1 x_2 +  j_3^{-2}    (j_2^2 + j_3^2) x_2^2,
\end{align}
for all rational numbers $x_1$ and $x_2$. Moreover, due to (\ref{xi1-2}) and (\ref{xi3}) and straightforward calculation, we obtain
\begin{align} \label{xixixi}
\xi_1 = \frac{n j_1}{|j|^2}, \;\;\; \xi_2 = \frac{n j_2}{|j|^2}, \;\;\;  \xi_3 = \frac{n^2}{|j|^2}.
\end{align}

Without loss of generality, we assume $d$ is an integer. We set 
\begin{align} \label{lcm}
\beta = \text{lcm}\{1^2, 2^2, 3^3, \cdots, d^2\}.
\end{align}

Consider arguments of the form $x_1 = \frac{i_1}{\beta}$ and $x_2 = \frac{i_2}{\beta}$. Then, by (\ref{Txx}), we obtain
\begin{align} \label{ae2}
&\left[T_{j,n} \Big(\frac{i_1}{\beta} + \xi_1, \frac{i_2}{\beta}+\xi_2\Big) - \xi_3\right] \beta^3 \notag\\
&= \beta j_3^{-2}(j_1^2 + j_3^2) i_1^2 + (2\beta j_3^{-2} j_1 j_2 ) i_1 i_2 + \beta j_3^{-2}(j_2^2 + j_3^2) i_2^2  \notag\\
&=: \tilde T_{j}(i_1, i_2).
\end{align}
Because of (\ref{lcm}), $\frac{i_1}{\beta} + \xi_1$ and $\frac{i_2}{\beta}+\xi_2$ can take any integer values by adjusting $i_1$ and $i_2$.
Also due to (\ref{lcm}) and (\ref{ae2}), $\tilde T_{j}(i_1, i_2)$ is a quadratic form with integer coefficients and negative discriminant. Note, $j$ belongs to a finite set. Thanks to Lemma \ref{thm1} and Remark \ref{remk1}, there exist arbitrary large $m_0$ and $h_0 \geq C \log m_0$ such that
\begin{align} \label{Ttilde}
\tilde T_{j}(i_1, i_2) \not \in [m_0, \, m_0+h_0], \;\text{for any} \;i_1,i_2 \in \mathbb Z, \; \text{and for any}\; j\in \mathbb Z^3 \; \text{with} \; 0<|j|\leq d.
\end{align}

For sufficiently large $h_0$, we can find $h>0$ satisfying
\begin{align} \label{h0}
&h_0 = \left(h + \frac{1}{4}(h + d^2)^2\right) \beta^3.
\end{align}
Then, we set
\begin{align} \label{defm}
m = \frac{m_0}{\beta^3} + \frac{1}{4}(h + d^2)^2. 
\end{align}

Due to (\ref{xixixi}) and (\ref{nbound}), we have
\begin{align}  \label{xi3b}
 \xi_3 = \frac{n^2}{|j|^2} \leq  \frac{1}{4}(h + d^2)^2.
  \end{align}

Using (\ref{ae2})-(\ref{xi3b}), we obtain
\begin{align} \label{Tmm}
T_{j,n} \Big(\frac{i_1}{\beta} + \xi_1, \frac{i_2}{\beta}+\xi_2\Big) 
\not \in [m - \frac{1}{4}(h + d^2)^2 + \xi_3, \, m+h + \xi_3] \supset [m, m+h],
\end{align}
for any $i_1,i_2 \in \mathbb Z$. In particular, there exist $i_1,i_2 \in \mathbb Z$ such that $k_1=\frac{i_1}{\beta} + \xi_1$ and $k_2=\frac{i_2}{\beta} + \xi_2$,
and thus (\ref{Tmm}) shows that $T_{j,n}(k_1, k_2) \not \in  [m, m+h]$, which contradicts (\ref{Tjnkk}). 
Therefore, for these pairs of $m$ and $h$, (\ref{lessd}) cannot happen. 
Thus, for any one of these pairs of $m$ and $h$,
if there exist two lattice points $k,\ell \in \mathbb Z^3$ that belong to the spherical shell 
$\{x\in \mathbb R^3: m \leq |x|^2 \leq m + h\}$, then $|k-\ell| > d$.

Recall that we fix $d$ at the beginning. Then, due to (\ref{h0}) and (\ref{defm}), asymptotically, 
\begin{align*}
h_0 \sim \frac{1}{4} \beta^3 h^2,  \;\;\;  m \sim \frac{1}{\beta^3} m_0.
\end{align*}
Thus, along with the fact that $h_0 \geq C \log m_0$, we conclude that
\begin{align*}
h \geq \tilde C \sqrt{\log m},
\end{align*}
for sufficiently large $m$, where the constant $\tilde C$ depends on $d$.
\end{proof}

\vspace{0.1 in}

\subsection{Proof of Proposition \ref{prop-new2}}
\begin{proof}
Let $C>4\sqrt[4]{8}$. For any sufficiently large real number $m>0$, 
it is easy to verify that
$\sqrt{m + C m^{1/8} - s} - \sqrt{m - s} > 2$,
for any $s\in \mathbb N$ satisfying $m- \beta m^{1/4}<  s < m$, where $2\sqrt{2} < \beta < \frac{C^2}{16}$.
Therefore, there exists a positive integer $n$ such that 
\begin{align}
\sqrt{m - s} \leq n < n+1 \leq \sqrt{m + C m^{1/8} - s}.
 \end{align}

By Bambah and Chowla \cite{Bambah}, for any large $m\in \mathbb R$, there exists $s=k^2 + l^2$, $k,l\in \mathbb Z$, such that $m- \beta m^{1/4}<  s < m$,
since $\beta>2\sqrt{2}$. Consequently, the spherical shell $\{x\in \mathbb R^3: m\leq |x|^2 \leq m +C m^{1/8}\}$ contains lattice points $(k,l,n)$ and $(k,l,n+1)$. 
\end{proof}

\vspace{0.1 in}

\subsection{Remarks}
Theorem \ref{thm3D} shows the existence of spherical shells
$\{x\in \mathbb R^3: m \leq |x|^2 \leq m + h\}$ in which lattice points are sparsely distributed. Here, $m$ can be arbitrarily large, and $h\geq C\sqrt{\log m}$.
The optimality of the order $\sqrt{\log m}$ is unknown, but it is closely connected to the logarithmic size of large gaps between sums of two squares due to 
Richards \cite{Richards}.
The proof of Theorem \ref{thm3D} relies on Lemma \ref{thm1}, which extends Richards' result to a family of quadratic forms. An important element of the proof is the prime number theorem.

Suppose Theorem \ref{thm3D} holds true for a higher order $h=h(m)$ with $\frac{h(m)}{\sqrt{\log m}} \rightarrow \infty$ as $m\rightarrow \infty$. Let's take $d=1$.
If $m$ is sufficiently large, then $\sqrt{m + h - s} - \sqrt{m - s} >2$,
for any $s \in \mathbb N$ satisfying $m-ch^2 < s < m$, where $0<c< \frac{1}{16}$. 
Thus, there exists a positive integer $n$ such that $\sqrt{m - s} \leq n<n+1\leq \sqrt{m + h - s}$.
If $s=k^2 + l^2$, then the spherical shell $\{x\in \mathbb R^3: m\leq |x|^2 \leq m +h\}$ contains lattice points $(k,l,n)$ and $(k,l,n+1)$. 
However, according to our assumption, there exist arbitrarily large values of $m$ such that any two lattice points in the spherical shell $\{x\in \mathbb R^3: m\leq |x|^2 \leq m +h\}$ must be separated by a distance strictly greater than $d=1$. 
Therefore, for these values of $m$, any $s \in (m-ch^2, m)$ cannot be expressed as a sum of two squares. This would improve the logarithmic size of large gaps between sums of squares, since we assume $\frac{h(m)}{\sqrt{\log m}} \rightarrow \infty$ as $m\rightarrow \infty$. However, as of now, there have been no advancements that surpass the logarithmic order discovered by Richards in \cite{Richards}.

\vspace{0.2 in}

\section{Discussion}
In this section, we discuss some related problems and future work.

The original motivation of this project was to find an optimal estimate for large gaps between sums of two squares, and this remains its long-term goal. It is an important open problem whether one can improve the logarithmic growth rate of the lower bound for large gaps between sums of squares presented in (\ref{bounds}).
Also, it is interesting to ask whether one can reduce the polynomial growth rate (the power of 1/4) of the upper bound in (\ref{bounds}).

We would like to draw some naive comparisons with a related problem: gaps between primes. It is well-known that the number of primes less than $x$ is approximately $\frac{x}{\log x}$, whereas the number of sums of squares less than $x$ behaves asymptotically as $\frac{bx}{\sqrt{\log x}}$, where $b$ is the Landau–Ramanujan constant. Therefore, there are fewer primes than sums of squares in $[0,x]$ for large $x$. Thus, on average, sums of squares are distributed more densely than primes throughout the natural numbers. It is a classical result of Westzynthius \cite {West} that
\begin{align} \label{gapprime}
\limsup_{n\rightarrow \infty} \frac{p_{n+1}-p_n}{\log p_n} = \infty,
\end{align}
 where $p_n$ is the sequence of primes. By comparing (\ref{gapprime}) and (\ref{loggap}),
we see that large gaps between primes might grow faster than large gaps between sums of squares asymptotically.
Also, it is worth mentioning that gaps between primes have an upper bound 
\begin{align} \label{primeupp}
p_{n+1} - p_n \leq p_n^{\theta},
\end{align}
with an estimate $\theta=0.525$, for sufficiently large $n$ (see \cite{Baker}). One can compare $\theta=0.525$ in (\ref{primeupp}) with the exponent $1/4$ in (\ref{bounds}) concerning the upper bound of gaps between sums of two squares. 
Furthermore, we would like to mention the twin prime conjecture regarding small gaps between primes; however, small gaps between sums of squares are always 1, which is trivial. 
All of the above observations show that, in general, sums of two squares appear more frequently than primes within the set of positive integers.
See \cite{Shiu} for more discussion on this topic.

Regarding the sparse distribution of lattice points in annuli in $\mathbb R^2$, we have already achieved an optimal asymptotic result concerning the width of these annuli in this paper. However, in three dimensions, the optimality of our estimates remains unknown, making it interesting to explore the possibility of refining the thickness of these spherical shells provided in Theorem \ref{thm3D}. This inquiry is closely related to the study of large gaps between sums of squares. Furthermore, our proofs contain deep geometric perspectives that can be explored further to generate other useful results. Most importantly, our findings on the sparse distribution of lattice points in annuli have the potential for applications in simplifying infinite-dimensional dissipative dynamical systems to finite-dimensional counterparts, particularly in solving the inertial manifold problem for the Navier-Stokes equations. For an example of such applications, please refer to \cite{MS}.

Sums of squares are eigenvalues of the Laplacian on a periodic domain. Likewise, we can consider gaps between eigenvalues of the Dirichlet Laplacian on a bounded domain in $\mathbb{R}^n$, or more generally, on an $n$-dimensional Riemannian manifold. In these cases, explicit expressions of eigenvalues are usually not available. An important problem is to find sharp estimates for the upper and lower bounds of the size of the gaps between eigenvalues of the Dirichlet Laplacian. Please refer to \cite{Chen} for an estimate of the upper bound.

\vspace{0.2 in}

\noindent {\bf Acknowledgments.} We extend our deep gratitude to the referee for offering insightful suggestions that have substantially improved the quality of our paper. In particular, Propositions 2.3 and 2.6 were developed based on the comments from the referee.

The research of this project was completed during the Summer 2023 REU program 
``AMRPU @ FIU" that took place at the Department of Mathematics and Statistics, Florida International University, 
and was supported by the NSF (REU Site) grant DMS-2050971. In particular, support of M. Ilyin came from the above grant.

\end{document}